\nonstopmode
\documentclass[12pt,a4paper]{article}
\usepackage{enumitem, pgf,tikz, layout, geometry, graphicx,  amssymb,amsmath,amsthm,bm,mathtools,faktor,wrapfig,comment}

\usepackage{microtype}
\usetikzlibrary{arrows,matrix}
\usepackage[utf8x]{inputenc}
\usepackage[colorinlistoftodos]{todonotes}
\usepackage[colorlinks=true, allcolors=blue]{hyperref}

\author{Johan Andersson\thanks{Email:johan.andersson@oru.se \, Address:Department of Mathematics, School of Science and Technology, {\"O}rebro University, {\"O}rebro, SE-701 82 Sweden. }}
 \date{}
\title{Joint universality
on the half plane of absolute convergence}

\theoremstyle{plain} 
\newtheorem{thm}{Theorem}  
\newtheorem{lem}{Lemma}
\newtheorem{cor}{Corollary} 
\theoremstyle{definition}

\newtheorem{defn}{Definition}

\newcommand{\C}{{\mathbb C}} 
\newcommand{\cH}{{\mathcal H}}
\newcommand{\R}{{\mathbb R}}

\newcommand{\Z}{{\mathbb Z}}

\newcommand{\norm}[1]{\left\lVert#1\right\rVert}
\newcommand{\abs}[1]{{\left| {#1} \right|}} \newcommand{\p}[1]{{\left(
     {#1} \right)}}

\renewcommand{\Re}{\operatorname{Re}}  
   
\begin{document}

\maketitle        
\begin{abstract} 
     We prove joint universality theorems 
     on the half plane of absolute convergence for general classes of Dirichlet series with an Euler-product, where in addition to vertical shifts we also allow scaling. This  generalizes our recent joint universality results for Dirichlet $L$-functions. In contrast to classical universality, we do not need that the Dirichlet series in question have an analytic continuation beyond their region of absolute convergence. Also we may allow weaker orthogonality conditions for pairs of Dirichlet series than in the previous joint universality results  of Lee-Nakamura-Pa\'{n}kowski. We take care to avoid using the Ramanujan conjecture in our proof and hence as a consequence of our universality theorem, we obtain stronger results on zeros of linear combinations of $L$-functions in the half plane of absolute convergence than previous results of Booker-Thorne and Righetti. For example as a consequence of our main universality result we have that certain linear combinations of Hecke $L$-series coming from  Maass wave forms have infinitely many zeros in any strip $1<\Re(s)<1+\delta$.
\end{abstract}

\maketitle
\section{Introduction and main results}
In a recent paper \cite{Andersson9}, 
 we introduced a new idea that allow us to prove a new type of universality theorem for a Dirichlet series with an Euler-product. In \cite{Andersson10} we further explored this idea and applied it to study universality results for the Hurwitz zeta-function. For rational parameters an important ingredient is a new type of joint universality result for Dirichlet $L$-functions \cite[Theorem 4]{Andersson10}, that also holds in the half plane of absolute convergence.  In this paper we will explore this idea further.
 \begin{defn} \label{def1}
   We say that a Dirichlet series 
   \begin{gather} \notag
  L(s)= \sum_{n=1}^\infty a(n) n^{-s}, \\ \intertext{is of standard type of order $(\lambda,\Lambda)$ if  it can be written as}
   L(s)= L_2(s) L_1(s)+L_3(s), \label{LLdef} \\ \intertext{where}  L_2(s)=1+\sum_{n=2}^\infty b(n) n^{-s}, \qquad L_3(s)=\sum_{n=1}^\infty d(n) n^{-s} \notag \\ \intertext{are absolutely convergent Dirichlet series for $\Re(s) \geq 1$,  and}
  L_1(s)=\sum_{n=1}^\infty c(n) n^{-s} =   \prod_{p} \sum_{k=0}^\infty c(p^k) p^{-ks} \notag 
\\ \intertext{has an Euler-product such that  }
  \notag \liminf_{N \to \infty} \sum_{N \leq p<N^{1+\xi}} \frac {|c(p)| } p \geq \lambda \log(1+\xi), \\ \intertext{and}
  \label{c}  \limsup_{N \to \infty} \sum_{N \leq p<N^{1+\xi}} \frac {|c(p)|^4 } p \leq \Lambda \log(1+\xi)
    \\ \intertext{for each $\xi>0$, and that} \label{c2}
  \sum_p \sum_{k=2}^\infty  \frac {|c(p^k)| k \log p } {p^k} <\infty, \qquad |c(p^k)|<p^k. 
\end{gather}
\end{defn}
So for example, by a weak version of the prime number theorem,  the Dirichlet $L$-functions are Dirichlet series of standard type and of order  $(1,1)$. We also define orthogonality of pairs of Dirichlet series of standard type
\begin{defn} \label{def2}
  Assume that the Dirichlet series
   \begin{gather*}
  A(s)= \sum_{n=1}^\infty a(n) n^{-s} 
  \qquad 
  \text{and} \qquad   B(s)= \sum_{n=1}^\infty b(n) n^{-s}
\\ \intertext{are of standard type. Then we say that the Dirichlet series $A(s)$ and $B(s)$ are orthogonal if for each $\xi>0$ then }
\lim_{N \to \infty} \abs{\sum_{ \substack{p \text{ prime} \\ N \leq p<N^{1+\xi}}} \frac{  a(p) \overline{b(p)}} p}=0, 
 \end{gather*}
 \end{defn}
Our definition of orthogonality is similar to  Righetti \cite[Definition 1.1]{Righetti}. It is easy to see that $A(s)$ and $B(s)$ are orthogonal if and only if $A_1(s)$ and $B_1(s)$ are orthogonal where $A_1(s),B_1(s)$ correspond to $L_1(s)$ in \eqref{LLdef}. Our joint universality theorem which proof follows in the same way as  \cite[Theorem 4]{Andersson10} is the following
     \begin{thm} \label{TH1} Let $L_1(s),\ldots,L_n(s)$ be pairwise orthogonal Dirichlet series of standard type. Let $K \subset  \{s \in \C : \Re(s)>0\}$ be a compact set with connected complement and let  $f_1,\ldots,f_n$ be continuous  functions on $K$ that are analytic in its interior.  Then for any $\varepsilon>0$ there exists some $C_0,\delta_0>0$, such that for any $|C_k|>C_0$ and  $0<\delta\leq \delta_0$ then
       \begin{gather*}
            \liminf_{T \to \infty} \frac 1 T \mathop{\rm meas} \left \{t \in [0,T]:\max_ {1 \leq k \leq n} \max_{s \in K} \abs{ L_k(1+it+\delta s)+C_k-f_k(s)}<\varepsilon \right \}>0. 
 \end{gather*} 
 \end{thm}
 Recently  Lee-Nakamura-Pa\'{n}kowski \cite{Lee}  proved a version of a conjecture of Steuding \cite{Steuding} about joint universality of functions of the Selberg class satisfying certain properties.  Although they have similar orthogonality conditions, it should be noted that our conditions on orthogonality are weaker than their orthogonality conjectures \cite[eq (4).]{Lee}. In particular by the prime number theorem
 \begin{gather*}
  \sum_{N \leq p<M} p^{ia-1} \approx \int_{N}^{M}
   \frac {x^{ia-1} dx} {\log x} = \left[\begin{matrix} y=\log x \\ x^{-1} dx = dy  \end{matrix}\right] = 
   \int_{\log N}^{ \log M} \frac {e^{iay}} y dy \ll
   \frac 1 {|a| \log N},
   \end{gather*}
 we have that for any real  $a \neq 0$ that
 \begin{gather*}
   \lim_{N \to \infty} \abs{\sum_{N \leq p<N^{1+\xi}} p^{ia-1}} =0,
 \end{gather*}
 and we have that  the Dirichlet series $\zeta(s+ia_1)$ and $\zeta(s+ia_2)$ are orthogonal for $a_1-a_2 =a \in \R \setminus\{0\}$.
    We may thus use the functions $L_k(s)=\zeta(s+ia_k)$ in Theorem \ref{TH1} and we obtain the following corollary for the Riemann zeta-function: 
 \begin{cor} \label{cr1}
     Let $a_1,\ldots,a_k$ be distinct real numbers. Let $K \subset  \{s \in \C : \Re(s)>0\}$ be a compact set with connected complement and let  $f_1,\ldots,f_n$ be continuous  functions on $K$ that are analytic in its interior.  Then for any $\varepsilon>0$ there exists some $C_0,\delta_0>0$, such that for any $|C_k|>C_0$ and  $0<\delta\leq \delta_0$ then
       \begin{gather*}
            \liminf_{T \to \infty} \frac 1 T \mathop{\rm meas} \left \{t \in [0,T]:\max_ {1 \leq k \leq n} \max_{s \in K} \abs{ \zeta(1+ia_k+it+\delta s)+C_k-f_k(s)}<\varepsilon \right \}>0. 
\end{gather*}
 \end{cor} 
 It is  clear that $\zeta(s+ia_1)$ and $\zeta(s+ia_2)$ are not jointly  universal in the classical sense, since we may choose the compact set $K$ sufficiently large so that both $s+ia_1 \in K$ and $s+ia_2 \in K$. 
 
 As a consequence of Theorem 1, we obtain in the same manner as we proved that \cite[Theorem 4]{Andersson10} implies \cite[Theorem 1]{Andersson10} the following theorem.
 \begin{thm}
 \label{TH2} Let $L(s)=\sum_{k=1}^n a_k L_k(s)$ be a linear combination of pairwise orthogonal Dirichlet series of standard type such that at least two of its coefficients $a_k$ are non-zero. Let $K \subset  \{s \in \C : \Re(s)>0\}$ be a compact set with connected complement and let $f$
 be a continuous  functions on $K$ that is analytic in its interior.  Then for any $\varepsilon>0$ there exists some $\delta_0>0$, such that for any  $0<\delta\leq \delta_0$ then
       \begin{gather*}
            \liminf_{T \to \infty} \frac 1 T \mathop{\rm meas} \left \{t \in [0,T]: \max_{s \in K} \abs{ L(1+it+\delta s)-f(s)}<\varepsilon \right \}>0. 
 \end{gather*} 
 \end{thm}
 \begin{proof}
    Define $a:=|a_1|^2+\cdots + |a_n|^2$  and 
    let $b_k$ be non-zero constants such that $b_1 a_1+\cdots+b_n a_n=0$. By Theorem \ref{TH1} the approximation
    \begin{gather*} 
    \max_ {1 \leq k \leq n} \max_{s \in K} \abs{ L_k(1+it+\delta s)+C b_k-\frac{\overline{a_k} f(s)} a}<\frac{\varepsilon} {a}
    \end{gather*}
        holds with positive lower density in $t$ for some sufficently large $C$. An application of the triangle inequality concludes the proof.
         \end{proof}
 By the observation that $\zeta(s)$ and $\zeta(s+ia)$ are  orthogonal if $a \neq 0$ is real, we obtain as a consequence of Theorem  \ref{TH2}  that
  \begin{cor} \label{cr2}
   Let $A \neq 0,B \neq 0$ be complex numbers and $a \neq 0$ be a real number. Let $K \subset  \{s \in \C : \Re(s)>0\}$ be a compact set with connected complement and let  $f$ be a continuous  function on $K$ that is analytic in its interior.  Then for any $\varepsilon>0$ there exists some $\delta_0>0$, such that  for any  $0<\delta\leq \delta_0$ then
       \begin{gather*}
            \liminf_{T \to \infty} \frac 1 T \mathop{\rm meas} \left \{t \in [0,T]: \max_{s \in K} \abs{ A\zeta(1+ia+it+\delta s)+  B\zeta(1+it+\delta s)-f(s)}<\varepsilon \right \}>0. 
\end{gather*}
\end{cor} 
 In Corollary \ref{cr1} and Corollary \ref{cr2} it is not actually neccessary to assume that $K$ lies in the half plane $\Re(s)>0$. This is because the Riemann zeta-function has a nice analytic continuation beyond the line $\Re(s)=1$. The reason why we include this condition is that in this case  we may state them as corollaries of Theorem \ref{TH1} and Theorem \ref{TH2}  which also holds for more general Dirichlet series. Another consequence of Theorem \ref{TH2} is that it gives a new proof  for the infinitude of zeros of linear combinations of Dirichlet series of Righetti \cite{Righetti0,Righetti} and Booker-Thorne \cite{BookerThorne} for $\Re(s)>1$. Since we do not assume the Ramanujan conjecture and it is known that \eqref{c} and \eqref{c2} holds for Maass wave-form $L$-functions, the same holds also for linear combinations of Maass wave-form $L$-functions\footnote{Booker-Thorne remarked that their method could also be extended to cover this case, see \cite[remark (3)]{BookerThorne}}. We remark that although in contrast to the result of Lee-Nakamura-Pa\'{n}kowski \cite{Lee}  we need to add constants in Theorem \ref{TH1}, in some of its important applications like Theorem \ref{TH2} this is no longer neccesary. Like in \cite{Andersson9} we may also state a version of Theorem \ref{TH1} when this is no longer neccesary. Then however we need additional restraints on the functions $f_k$. We will find it convenient to assume the somewhat stricter condition that the Dirichlet series in question have an Euler product, but in addition we will also prove a version of so called {\em hybrid universality theorem}. 
  \begin{thm} \label{TH3}  Let $|a_p|=1$ for primes $2 \leq p \leq N$, let $L_1(s),\ldots,L_n(s)$ be pairwise orthogonal Dirichlet series of standard type,  where  $L_k(s)$ are of order $(\lambda,\Lambda)$ and have Euler products\footnote{This means that $L_2(s)=1$ and $L_3(s)=0$ in \eqref{LLdef}.}. Let $K \subset  \{s \in \C   :  \Re(s)>0\}$ be a compact set and let  $f_1,\ldots,f_n$ be continuous  functions on $K$ that are analytic in its interior, where
  \begin{gather*}
   f_k(s)= C_k+\int_0^{\infty} g_k(x)e^{-sx} dx,  \qquad (s \in K) 
 \end{gather*}
 are Laplace-transforms of the functions $g_k(x)$  so that
 \begin{gather}
          \abs{x g_k(x)} \leq \frac 1 8 \sqrt{\frac{\lambda^{3}}{n^3 \Lambda}  },  \qquad (k=1,\ldots,n \,; \, x \geq 0). \label{xgkx} 
       \end{gather}
    Then for any  $\varepsilon>0$ there exist some  $\delta_0>0$  such that for each $0<\delta\leq \delta_0$  then
       \begin{gather*}
            \liminf_{T \to \infty} \frac 1 T \mathop{\rm meas} \left \{t \in [0,T] : \begin{matrix*}[l] \max_{2 \leq p \leq N} \abs{p^{it}-a_p}<\varepsilon, \\ \max_ {1 \leq k \leq n} \max_{s \in K} \abs{\log L_k(1+it+\delta s)-f_k(s) } <\varepsilon \end{matrix*} \right \}>0. 
 \end{gather*} \end{thm}

 \section{Proof of Theorem \ref{TH3}}

  Like in \cite{Andersson9}, Theorem \ref{TH1} follows from Theorem \ref{TH3} (see section \ref{sec3}), so we will start by proving Theorem \ref{TH3}.
 We need the following fundamental  lemma (for its proof, see section \ref{sec4}) which replaces the Pechersky rearrangement theorem in classical universality proofs.
 \begin{lem}
 \label{Le2}
 For any $N,\varepsilon>0$, numbers $|a_p|=1$ for primes $2\leq p \leq N$, compact set $K \subset \C$, coefficients $c_1,\ldots,c_k$   of pairwise orthogonal Dirichlet series of standard type and with Euler products, and functions  $f_1,\ldots,f_n$,  satisfying the conditions of Theorem \ref{TH3}  there exists some $\delta_0>0$ such that for any $0<\delta \leq \delta_0$ there exists some 
  completely multiplicative unimodular function $\omega:\Z^+ \to \C$ such that $\omega(p)=a_p$ for the primes $p \leq N$ and such that 
  \begin{gather*}
     \max_{k=1,\ldots,n} \max_{s \in K} \abs{ 
\sum_{p} \frac{\omega(p) c_k(p)} {p^{1+\delta s}}-f_k(s) }<\varepsilon. 
  \end{gather*}
 \end{lem}

 We are now ready to prove Theorem \ref{TH3}. 
  We will let $\omega:\Z^+ \to \C$ denote a completely multiplicative unimodular function and define
 \begin{gather*} 
    L_k(s,\omega):=\sum_{n=1}^\infty \frac{\omega(n) c_k(n)} {n^{s}}, \\ \intertext{and}
    E_k(s,\omega):=\exp \left(\sum_{p}  \frac{\omega(p) c_k(p)} {p^{s}} \right) \frac 1 {L_k(s,\omega)} = \sum_{j=1}^\infty \frac{\omega(j) e_k(j)} {j^{s}}.
    \end{gather*}
    It is clear that $E_k(s,\omega)$ is an zero-free and  absolutely convergent Dirichlet series for $\Re(s) \geq 1$. 
     Let $\omega_1(p)=a_p$ for primes $2 \leq p \leq N$ and $\omega_1(p)=1$ for primes $p>N$.   Since $E_k(1,\omega_1) \neq 0$ is absolutely convergent there exists some $N_0 \geq N$ such that for any $|\omega_2(n)|=1$ where $\omega_2(n)=\omega_1(n)$ for $n \leq N_0$ we have that
     \begin{gather}  
     \label{tra2aa}
      \abs{\log E_k(1,\omega_1) - \log \p{\sum_{j=1}^{N_0} \frac {\omega_2(j)e_k(j) } j}} <\frac{\varepsilon} 4. 
      \end{gather}
     Define 
    \begin{gather*}
      D_k:=C_k+\log E_k(1,\omega_1).
      \end{gather*}
      By Lemma \ref{Le2} we may find a unimodular completely multiplicative function  $\omega_2:\Z^+ \to \C$  such that $\omega_2(p)=\omega_1(p)$ if $p \leq N_0$ and such that
      \begin{gather}  \label{tra1}
       \max_{1 \leq k \leq n} \max_{s \in K}  \abs{\sum_{p}  \frac{\omega_2(p) c_k(p)} {p^{1+\delta s}} -f_k(s)-D_k} <\frac{\varepsilon} 4.
      \end{gather}
      Since 
      $L_k(s,\omega_2)$ is absolutely convergent for $\Re(s)>1$ we may now choose $N_1>N_0$ sufficiently large such that
      \begin{gather} 
       \label{tra2a}
      \sum_{ N_1 \leq p} \frac{\abs{c_k(p)}} {p^{1+\delta \xi}} <\frac{\varepsilon} 4, \qquad \text{where} \qquad \xi=\min_{s \in K} \Re(s)>0.
      \end{gather}
      It is clear that if 
     \begin{gather} \label{ire5}\max_{2 \leq p < N_1} \abs{p^{-it}-\omega_2(p)}<\varepsilon_2. \end{gather}
      for some sufficiently small $\varepsilon_2>0$ 
         then 
      \begin{gather} \label{tra2}
       \max_{1 \leq k \leq n} \max_{s \in K}  \abs{\sum_{p<N_1}  \frac{\omega_2(p) c_k(p)} {p^{1+\delta s}} -\sum_{p<N_1}  \frac{c_k(p)} {p^{1+it+\delta s}} } <\frac{\varepsilon} 4.
      \end{gather}
It follows by  \eqref{tra2aa}, \eqref{tra1}, \eqref{tra2a}, \eqref{tra2} and the triangle inequality that
\begin{gather*}
\max_{2 \leq p \leq N} \abs{p^{it}-a_p}<\varepsilon, \, \,  \text{and} \, \, \max_ {1 \leq k \leq n} \max_{s \in K} \abs{\log L_k(1+it+\delta s)-f_k(s)} <\varepsilon,
\end{gather*}
whenever $t$ satisfies the inequalities \eqref{ire5} for some sufficiently small $\varepsilon_2>0$. By Weyl's version (see \cite{Weyl} or  \cite[Lemma 1.8]{Steuding}) of the Kroenecker's approximation theorem the set of $0 \leq t \leq T$ where \eqref{ire5} holds has positive lower density as $T \to \infty$. \qed

 \section{Proof that Theorem \ref{TH3} implies Theorem \ref{TH1} \label{sec3}}

The next lemma appears as
\cite[Lemma 1]{Andersson9} and is a consequence of Mergelyan's theorem and the theory of Laplace transforms. 
\begin{lem} \label{LA3} Assume that $f$ is any zero-free function on a compact set $K$ with connected complement, that is analytic in the interior of $K$. Then given $\varepsilon>0$ there exist some $A,B,N>0$  and continuous function $g:[A,B] \to \C$ such that $|g(x)| \leq N$ and that if
\begin{gather*}
  G(s) = \int_A^B g(x) e^{-sx} dx, \\ \intertext{then}
  \max_{s \in K} \abs{G(s)-f(s)}<\varepsilon,
\end{gather*}
\end{lem}
\begin{proof} See \cite[Lemma 1]{Andersson9}. \end{proof}

\noindent{\em Proof of Theorem \ref{TH1}.}
Let
 \begin{gather*}
  L_k(s)=L_{1,k}(s) L_{2,k}(s)+L_{3,k}(s) 
\end{gather*}
where \begin{gather*}
 L_{1,k}(s)=\sum_{j=1}^\infty c_k(j) j^{-s} = \prod_{p} \sum_{l=0}^\infty c_k(p^l) p^{-ls}
\end{gather*}
has an Euler product and the 
 Dirichlet series \begin{gather*} L_{2,k}(s)=1+\sum_{j=2}^\infty b_k(j)  j^{-s}, \qquad \text{and} \qquad  L_{3,k}(s)=\sum_{j=1}^\infty  d_k(j) j^{-s} \end{gather*} are absolutely convergent for $\Re(s) \geq 1$. Since $L_{2,k}(s)$ for each $1 \leq k \leq n$ are absolutely convergent and not identically zero for $\Re(s) \geq 1$ there must exist some $t_0 \in \R$ such that 
 \begin{gather}  \label{sr1}
 \min_{k=1,\ldots,n} \left| L_{2,k}(1+it_0) \right|=\xi>0.\end{gather}
Let 
\begin{gather} \label{Fks}  F_k(s)=f_k(s)-L_{3,k}(1+it_0).
\end{gather}
By Lemma \ref{LA3} there exist some $0<A<B$ and continuous functions  $g_k:[A,B] \to \C$ such that
\begin{gather} \label{ab0} 
 \max_{s \in K}  \abs{G_k(s)-    F_k(s)}<\frac{\varepsilon} 8,  \\ \intertext{where} 
    G_k(s)=\int_A^B g_k(x) e^{-sx} dx. \label{Gks}
\end{gather}
Let \begin{gather} \label{Ck} |C_k| \geq  8  n^{3/2}\Lambda^{1/2} \lambda^{-3/2} \max_{A \leq x \leq B} |x g_k(x)| \end{gather}  be given.  
Since   $L_{2,k}(s)$ 
 and $L_{3,k}(s)$ are absolutely convergent for $\Re(s)\geq 1$ it follows from \eqref{sr1} that for any $\varepsilon>0$ there exists some $N \in \Z^+$, $\delta_1>0$  and $\varepsilon_1>0$ such that if
\begin{gather} \label{ab3a}
  \max_{2 \leq p \leq N} |p^{-it}-p^{-it_0}|<\varepsilon_1, 
\end{gather}
and $0 <\delta \leq \delta_1$ then
  \begin{gather}\label{ab2j} \max_{1 \leq k \leq n} \max_{s \in K} \abs{\log(L_{2,k}(1+\delta s+it))-\log(L_{2,k}(1+it_0))}<\frac{\varepsilon}{8|C_k|}, \end{gather}
  and
   \begin{gather} \label{uit1} \max_{1 \leq k \leq n} \max_{s \in K} \abs{L_{3,k}(1+\delta s+it)-L_{3,k}(1+it_0))}<\frac \varepsilon 8. \end{gather}
  The conditions \eqref{Gks} and \eqref{Ck} ensures us that $G_k(s)/C_k$ satisfies the condition of  Theorem \ref{TH3}  so that we may use\footnote{It is clear that if $L_k$ are jointly orthogonal then also $L_{1,k}$ are jointly orthogonal, see remark under Definition \ref{def2}.} Theorem \ref{TH3} with $a_p=p^{-it_0}$ for $p \leq N$ and  there exists some $\delta_2 >0$ such that \eqref{ab3a} (and thus also \eqref{ab2j} and \eqref{uit1}) and 
\begin{gather} \label{ab2}
  \max_{s \in K} \abs{\log L_{1,k}(1+it+\delta s)-\frac{G_k(s)} {C_k}+\log L_{2,k}(1+it_0)- \log (-C_k)}<\frac {\varepsilon} {8|C_k|}
\end{gather}
holds for any  $0<\delta \leq \delta_2$  with a  positive lower measure $0\leq t \leq T$ as $T \to \infty$.
It is clear that
\begin{gather} \label{iden2}
   \log(G_k(s)-C_k)=\log \p{-C_k \p{1- \frac {G_k(s)}  {C_k}}}= \log (-C_k)+\log \p{1-\frac {G_k(s)}   {C_k}}.
\end{gather}
Since 
\begin{gather} \label{ab11}
\max_{1 \leq k \leq n} \max_{s   \in K} \abs{\frac{G_k(s)} {C_k} +\log \p{1-\frac {G_k(s)} {C_k}}}<\frac{\varepsilon} {8|C_k|},
\end{gather}
for $C_k$ sufficiently large, it follows by  \eqref{ab2j}, \eqref{ab2}, \eqref{ab11} and the triangle inequality that 
 \begin{gather} \label{ab23}
  \max_{s \in K} \abs{\log \p{L_{1,k}(1+it+\delta s) L_{2,k}(1+it+\delta s)}- \log(G_k(s)-C_k)}<\frac {3\varepsilon} {8|C_k|}
\end{gather}
holds as well as \eqref{uit1} holds for $0<\delta \leq \delta_0:=\min(\delta_1,\delta_2)$ with positive lower measure  $0 \leq t \leq T$ as $T \to \infty$.
From the elementary inequality $|e^z-1| \leq  3|z|/2$ if $|z| \leq 1/2$ it follows that
\begin{gather} \label{rrra}
  |X-Y| = \abs{Y} \cdot \abs{e^{\log(X/Y)} -1} < 3\abs{Y}/2 \cdot \abs{\log X-\log Y},
\end{gather}
when $\abs{\log X-\log Y} \leq 1/2$. Now let $X=L_{1,k}(1+it+\delta s) L_{2,k}(1+it+\delta s)$ and $Y= G_k(s)-C_k$. From \eqref{Ck} it follows that $|Y| \leq 5|C_k|/4$ for $s \in K$, and from \eqref{ab23} it follows that $\abs{\log X-\log Y} \leq 1/2$. Thus from \eqref{ab23} and \eqref{rrra} it follows that
\begin{multline}
  \max_{s \in K} \abs{L_{1,k}(1+it+\delta s) L_{2,k}(1+it+\delta s) +C_k-G_k(s)}  < \frac 3 2   \cdot \frac {5|C_k|} 4 \cdot \frac{3\varepsilon}{8|C_k|} <\frac {3\varepsilon} 4 \label{ajaj}  
\end{multline}
holds for any  $\delta,t$ such that \eqref{ab2} and \eqref{ab2j}  holds. By   \eqref{Fks}, \eqref{ab0}, \eqref{uit1},  \eqref{ajaj}, and the triangle inequality it follows that
\begin{gather*}
  \max_{s \in K} \abs{L_k(1+it+\delta s) +C_k-f_k(s)}<\varepsilon,
\end{gather*}
holds when $0<\delta\leq \delta_0$ with positive lower measure  $0 \leq t \leq T$ as $T \to \infty$. \qed

 \section{Proof of Lemma \ref{Le2} \label{sec4}} 
 \subsection{Some preliminary Lemmas}
 We are going to need the following Lemma \cite[Lemma 5.2]{Steuding} on Hilbert-spaces\footnote{This Lemma has important application on the classical Voronin universality theorem since it together with the Riesz representation theorem is used to prove the Pechersky rearrangement theorem which is the standard tool to prove classical universality.}
 
 \begin{lem} \label{LE12} Let $x_1,\ldots,x_m$ be points in a complex Hilbert space $\mathcal H$  and let $a_1,\ldots,a_m$ be complex numbers with $|a_j| \leq 1$ for $1\leq j \leq m$. Then there exist complex numbers $b_1,\ldots,b_m$ with $|b_j|=1$ for $1\leq j \leq n$ satisfying the inequality
 \begin{gather*}
   \norm{ \sum_{j=1}^m a_j x_j -\sum_{j=1}^m b_j x_j}^2 \leq 4 \sum_{j=1}^m \norm{x_j}^2
 \end{gather*}
 \end{lem}
 In our approach it is sufficient to apply this result on the Hilbert space $\cH=\C^n$. Assuming a Ramanujan-Petersson conjecture for the Dirichlet series $L_k(s)$ would be convenient at this stage but the next Lemma shows that \eqref{c}  is a sufficent replacement.

   \begin{lem} \label{LE5}
     Let $\bm{v} \in \C^n$ be such that $\bm{v} \neq \bm{0}$. 
      Let $\bm{c}(p):=p^{-1}(c_1(p),\ldots,c_n(p))$ where $c_1(p),\ldots,c_n(p)$ are the coeffients of the pairwise orthogonal Dirichlet series $L_1(s),\ldots,L_n(s)$ of standard type and order  $(\lambda,\Lambda)$. Then  there exist for any sufficiently large $N$ some coefficients $|d_p| \leq 1$ and  some  $N_0>N$ such that
     \begin{gather} \notag
       \norm{\bm{v} - \sum_{N \leq p<N_0}  d_p  \bm{c}(p)}_2 \leq \p{1-\frac 1 {4n}}\norm{\bm{v}}_2, 
     \\ \intertext{where furthermore we may choose $N_0$ as}
         N_0 \leq N^{\exp \p{2 (1-(8n)^{-1})^{-1} \sqrt {n \Lambda/\lambda^3} \, \norm{\bm{v}}_\infty}}  \label{N1c}
        \end{gather}
   \end{lem}

      \begin{proof} 
         By \eqref{c} we find that
               \begin{gather} \label{evert} \sum_{\substack{N \leq p<N^{1+\xi} \\ |c_k(p)| \geq B}}   \frac  1 p  \leq \frac {\Lambda \log(1+\xi)} {B^4}+o(1), \qquad (\xi>0),
               \end{gather}  
                     where the little ordo notation, as in the rest of the proof, is with respect to $N$.
                 By a repeated application of the Cauchy-Schwarz inequality and then using \eqref{c} we get 
                 \begin{align*}   
            \abs{ \sum_{\substack{N \leq p<N^{1+\xi} \\ |c_k(p)| \geq B}} \frac{c_k(p)c_j(p)} p}^2  &\leq \sum_{\substack{N\leq p<N^{1+\xi} \\ |c_k(p)| \geq B}} \frac{\abs{c_k(p)}^2} p  \sum_{\substack{N \leq p<N^{1+\xi} \\ |c_k(p)| \geq B}} \frac{\abs{c_j(p)}^2} p,  \\ 
             &\leq \sqrt{\sum_{\substack{N \leq p<N^{1+\xi} \\ |c_k(p)| \geq B}} \frac{\abs{c_k(p)}^4} p }\sqrt{\sum_{\substack{N \leq p<N^{1+\xi} \\ |c_k(p)| \geq B}} \frac{\abs{c_j(p)}^4} p }   \sum_{\substack{N \leq p<N^{1+\xi} \\ |c_k(p)| \geq B}} \frac 1 p,
             \\   &\leq \p{\Lambda \log(1+\xi)+o(1)}  \sum_{\substack{N \leq p<N^{1+\xi} \\ |c_k(p)| \geq B}}\frac 1 p, \\ &\leq \frac{\p{\Lambda \log(1+\xi)}^2}{B^4}+o(1),  
               \end{align*}
         where the last inequality follows by \eqref{evert}. Thus by taking the square roots we have
      \begin{gather} \label{yut}
          \abs{\sum_{\substack{N \leq p<N^{1+\xi} \\ |c_k(p)| \geq B}} \frac{c_k(p)c_j(p)} p}  \leq    \frac{\Lambda \log(1+\xi)} {B^2}+o(1). 
                 \end{gather}
                 By the Cauchy-Schwarz inequality and the inequalities \eqref{evert} and \eqref{yut} we obtain
                 \begin{gather} \label{yut2}
          \abs{\sum_{\substack{N \leq p<N^{1+\xi} \\ |c_k(p)| \geq B}} \frac{c_k(p)} p}  \leq    \frac{\Lambda \log(1+\xi)} {B^3}+o(1). 
                 \end{gather}
                                  By  the fact that $L_{k}(s)$ is of order $(\lambda,\Lambda)$ and  \eqref{yut2}, \eqref{Bdef} we get that
 \begin{gather} \label{ajjoj}
                     \sum_{\substack{N \leq p<N^{1+\xi} \\ |c_{k}(p)| \leq B}} \frac{\abs{c_{k}(p)}} {p}>\p{\lambda-\frac{\Lambda} {B^3}} \log(1+\xi)+o(1), 
                 \end{gather}
for $\xi>0$.  In the rest of the proof we will let $1 \leq k \leq n$ be chosen such that \begin{gather} \notag 
                \abs{v_k}=\norm{\bm{v}}_\infty. \\ \intertext{Define  $B$ by}  \label{Bdef} B:= 2 \sqrt{\frac{n \Lambda} \lambda}.
                \end{gather}
        By \eqref{ajjoj}  with this choice of $B$ we get that                                        \begin{gather} \label{ajj}
                     \sum_{\substack{N \leq p<N^{1+\xi} \\ |c_{k}(p)| \leq B}} \frac{\abs{c_{k}(p)}} {p} \geq  \p{1-\frac 1 {8n}} \lambda \log(1+\xi)+o(1), \qquad
                 \end{gather}
                 for $\xi>0$. Define \begin{gather}
                           \label{bpdef} 
                      d_p:=\begin{cases}  \exp(i \arg v_{k}) \frac{ \overline{c_{k}(p)}} B,  & |c_{k}(p)| \leq B, \\ 0, & \text{otherwise,} \end{cases}
                  \end{gather}
              and $N_0>N$ that minimizes
              \begin{gather} \notag
               \abs{v_{k}-\sum_{N \leq p<N_0} \frac{d_p c_{k}(p)} p}. \\ \intertext{Let $\xi$ be defined so that } \label{N0def}  N_0=N^{1+\xi}. \end{gather}
          It follows by \eqref{ajj} and the definition of $N_0$ that
          \begin{gather}
             \frac 1 B \p{1-\frac 1 {8n}} \lambda \log(1+\xi)\leq  \norm{\bm{v}}_\infty+o(1), \notag \\ \intertext{and thus} \label{yurt} \begin{split} \log (1+\xi) &\leq  \p{1-\frac 1 {8n}}^{-1} \frac{B \norm{\bm{v}}_\infty}{\lambda}+o(1), \\ &= 2 \p{1-\frac 1 {8n}}^{-1}  \sqrt{\frac{n \Lambda} {\lambda^3}}\,  \norm{\bm{v}}_\infty+o(1). \end{split}
          \end{gather}
         Define \begin{gather} \notag
     \bm{w}:= \sum_{N \leq p<N_0}   d_p \bm{c}(p). 
     \end{gather}  
       We get the inequality \eqref{N1c} from  \eqref{N0def}, \eqref{yurt} and our assumption.
     By \eqref{c}, \eqref{ajj},  \eqref{bpdef} and \eqref{N0def} it follows that
     \begin{gather}
         \abs{v_{k}-w_{k}}=o(1). \label{AA1}
     \end{gather}
      By \eqref{yut} and \eqref{yurt} it follows that
     \begin{gather} \begin{split} 
      \abs{w_j}&< \frac{\Lambda \log(1+\xi_0)}{B^3} + 
      o(1), \\ &\leq \p{1-\frac 1 {8n}}^{-1}  \frac{\norm{\bm{v}}_\infty \Lambda}{ \lambda B^2} + 
      o(1), \\ &\leq  \p{1-\frac 1 {8n}}^{-1} \frac {\norm{\bm{v}}_\infty} {4n} +o(1), \end{split} \label{AA2}
      \end{gather}
      where the last inequality follows from \eqref{Bdef}.
         It follows from \eqref{AA1} and \eqref{AA2} that 
      \begin{gather} \notag \begin{split}
     \norm{\bm{v}-\bm{w}}_2^2 &\leq  \sum_{j \neq k} \abs{v_j-w_j}^2 +o(1), 
      \\ &\leq \sum_{j \neq k} \abs{v_j}^2+  \sum_{j \neq k} \abs{2 v_j w_j} +\sum_{j \neq k} \abs{w_j}^2 + o(1),   \\   &\leq \norm{\bm{v}}_2^2- \norm{\bm{v}}_\infty^2+\frac {n-1} {2n-1/4}   \norm{\bm{v}}_\infty ^2 +\frac{n-1}{(4n-1/2)^2} \norm{\bm{v}}_\infty^2
         +o(1), \\  & \leq  \norm{\bm{v}}_2^2- \frac{\norm{\bm{v}}_\infty^2} 2+o(1), \\ &\leq \p{1-\frac 1{2n}} \norm{\bm{v}}_2^2  +o(1), \end{split} \\ \intertext{and thus by taking the square roots} \notag 
     \norm{\bm{v}-\bm{w}}_2 \leq \sqrt{1-\frac 1  {2n}} \norm{\bm{v}}_2+o(1) , 
     \end{gather}
     and by using the elementary strict inequality $$\sqrt{1-\frac 1 {2n}} <1-\frac 1 {4n}, \qquad (n \geq 1), $$ it follows that if $N$ is sufficiently large then
     \begin{gather*}
      \norm{\bm{v}-\bm{w}}_2 \leq \p{1-\frac 1  {4n}} \norm{\bm{v}}_2. 
     \end{gather*}   
            \end{proof}

 \begin{lem} \label{lem5}
  Let $L_1(s),\ldots,L_n(s)$  be pairwise orthogonal Dirichlet series of standard type and of order $(\lambda,\Lambda)$, where
   \begin{gather*}
     L_k(s)=\sum_{j=1}^\infty \frac{c_k(j)}{ j^{s}}.
    \end{gather*}
     Then for any $\xi>0$ and $\varepsilon>0$ there exists some $P_0>0$ such that if  $b_k$ satisfies 
      \begin{gather} \label{hury}
      \abs{b_k} \leq \frac 1 8 \sqrt{\frac{\lambda^3}{n^{3} \Lambda}}, \qquad (k=1,\ldots,n),
       \end{gather}
       and $P>P_0$, then  there exists some unimodular numbers $|\omega(p)|=1$ for primes $P\leq p<P^{1+\xi}$  such that
       \begin{gather*}
          \max_{1\leq k \leq n}\abs{\frac 1 {\log(1+\xi)} \sum_{P \leq p<P^{1+\xi}}\frac{\omega(p) c_k(p)} p -b_k} <\varepsilon.
\end{gather*}
\end{lem}
\begin{proof}
 Define \begin{gather}\bm{v_0}:=\log(1+\xi) (b_1,\ldots,b_n), \qquad  \bm{c}(p):=p^{-1}(c_1(p),\ldots,c_n(p)), \qquad P_1=P, \label{v1def} \end{gather} and use  Lemma \ref{LE5} to define $|d_p| \leq 1$ and 
 \begin{gather*}
     \bm{v}_k:=\bm{v}_{k-1}+ \sum_{P_k \leq p<P_{k+1}} d_p \bm{c}(p), \\ \intertext{recursively for $k=1,2,\ldots$ such that}
    \norm{\bm{v}_{k}}_2 \leq \p{1-\frac 1 {4n}}^{k} \norm{\bm{v}_0}_2.
   \end{gather*}
   Since the right-hand side of the inequality tend to zero as $k \to \infty$ we may now choose $k$ sufficiently large such that
  \begin{gather} \notag  \norm{\bm{v}_{k}}_2 \leq \frac{\varepsilon \log(1+\xi)} {2 \sqrt n}, \\ \intertext{and we have that} \label{rot1a}
   \norm{\bm{v}_0-\sum_{P \leq p<P_{k}} d_p \bm{c}(p)  }_2 \leq \frac{\varepsilon \log(1+\xi)} {2 \sqrt n}.
     \end{gather}
     It remains to estimate $P_{k}$ by using \eqref{N1c}. It follows that
      $P_{k}=P^{1+\xi_0}$ where
      \begin{align*} 1+\xi_0 &<\exp \p{ \sum_{k=1}^\infty  \left(1-\frac 1 {8n} \right)^{-1} \p{1-\frac 1 {4n}}^k 2 \sqrt{\frac{ n \Lambda}{\lambda^3}}  \norm{\bm{v}_0}_\infty },  \\ &= \exp \p{\p{1 - \frac 1 {4n}} \p{1-\frac 1 {8n}}^{-1} 8n^{3/2} \Lambda^{1/2} \lambda^{-3/2} \norm{\bm{v}_0}_\infty}, \\ &<  \exp \p{8n^{3/2} \Lambda^{1/2} \lambda^{-3/2} \norm{\bm{v}_0}_\infty}, \\ &\leq
      \exp \p{\log(1+\xi)}=1+\xi,
      \end{align*}
      where the last inequality follows by \eqref{hury} and \eqref{v1def}.
      Thus $P<P_{k}<P^{1+\xi}$. Define $d_p:=0$ for $P_{k}<p<P^{1+\xi}$. 
       By Lemma \ref{LE12} we may find some unimodular numbers $|\omega(p)|=1$ such that
     \begin{gather} \label{rot2}
        \norm{\sum_{P \leq p<P^{1+\xi}} d_p \bm{c}(p) - \sum_{P\leq p<P^{1+\xi}} \omega(p) \bm{c}(p)}_2 \leq 2 \sqrt{\sum_{P\leq p<P^{1+\xi}} \frac 1  {p^2}}< \frac {\varepsilon \log(1+\xi)}{2 \sqrt n},
        \end{gather}
        where the final inequality is true for  $P$ sufficently large.
      By the norm-inequality $\norm{\bm{u}}_\infty \leq \sqrt n \norm{\bm{u}}_2$, the inequalities \eqref{rot1a}, \eqref{rot2} and the triangle inequality it follows that
      $$
       \norm{\bm{v}_0 - \sum_{P\leq p<P^{1+\xi}} \omega(p) \bm{c}(p) }_\infty <\varepsilon \, \log(1+\xi).
     $$
     By recalling the definition \eqref{v1def} and by dividing both sides of this inequality by $\log(1+\xi)$ we obtain our result.
    \end{proof}

The following Lemma may be proved directly but is also a consequence of lemma \ref{lem5}
\begin{lem} \label{lem6}
  Let $a_k$ for $k=1,\ldots,n$, $N_0$ be defined as in Lemma \ref{lem5}. Then given any constants $C_1,\ldots,C_n$ and $\varepsilon>0$ there exists some $P_0$ such that for any $P_1>P_0$ we can find some unimodular numbers $|\omega(p)|=1$ for primes $p<P_1$ where $\omega(p)=a_p$ if $p<N_0$ such that
  \begin{gather*}
     \max_{1\leq k \leq n}\abs{ \sum_{p <P} \frac{\omega(p) c_k(p)} p- C_k} <\varepsilon. 
     \end{gather*}
  \end{lem}
 
\begin{proof} 
 We follow the proof of \cite[Lemma 7]{Andersson9}. Choose\footnote{It follows if $\omega_0$ is chosen randomly a suitable sense. } $|\omega_0(p)|=1$ such that $\omega_0(p)=a_p$ for $p\leq N$ and such that the series 
 \begin{gather*}
  \sum_{p}\frac{\omega_0(p) c_k(p) } p =D_k
   \end{gather*}
are convergent for each $k=1,\ldots,n$.   Let us now define 
   $$
    E_k:=\frac{C_k+D_k} N,
   $$
   where $N \in \Z^+$ is sufficiently large so that $$|E_k|< \frac{\lambda^{3/2} \log 2}{8 \Lambda^{1/2} n^{3/2}}$$ for $k=1,\ldots,n$. Let  $P_1>N$ be sufficiently large such  that if $P\geq P_1$ then
      \begin{gather}\label{uri}
   \max_{1\leq k \leq n} \abs{\sum_{p < P} \frac{\omega_0(p) c_k(p)} p -D_k}<\frac {\varepsilon} 2,  
           \end{gather}
   and such that for each $Q \geq P_1$ we can use Lemma \ref{lem5} with $\xi=1$  to define $|\omega(p)|=1$ for $Q\leq p < Q^2$ such that
   \begin{gather} \label{ure}
       \max_{1\leq k \leq n}\abs{\sum_{Q \leq p < Q^2}\frac{\omega(p) c_k(p)} p -E_k} <\frac{\varepsilon}{2N}.
   \end{gather}
   By defining $\omega(p)=\omega_0(p)$ for $p < P$ and by \eqref{ure} for $Q \leq p <Q^2$ for  $Q=P^{2^j}$ for each $j=0,\ldots,N-1$ then the conclusion of the lemma follows with $P_0=P_1^{2^N}$ by the inequalities \eqref{uri}, \eqref{ure} and the triangle inequality. 
    \end{proof}

\subsection{Proof of Lemma \ref{Le2}.}  
 
 We follow the proof of \cite[Lemma 2]{Andersson9}. 
 Choose\footnote{The convergence is well-known for ``almost all'' $\omega_0$ in a suitable sense. By Carleson's theorem we may even choose $\omega_0(p)=e^{2 \pi i p x}$ for almost all $0\leq x \leq 1$} $|\omega_0(p)|=1$ such that
 \begin{gather*}
   A_k(s)=\sum_p \frac{\omega_0(p) c_k(p)} {p^s}
 \end{gather*}  
  are convergent to analytic functions on the half plane $\Re(s)>1/2$ for each $k=1,\ldots,n$. Choose $P_0 \geq N$ and $\delta_1$ sufficiently small such that $1+\delta_1 K \subset \{s \in \C: \Re(s)>3/4\}$ and such that
  \begin{gather} \label{oo3}
     \sup_{P \geq P_0} \max_{0 \leq \delta \leq \delta_1} \max_{s \in K} \sum_{s \in K} \abs{\sum_{P \leq p} \frac{\omega_0(p) c_k(p)}{p^{1+\delta s}}}<\frac{\varepsilon} 9.
  \end{gather}
By Lemma \ref{lem6} we may find some $P_1\geq P_0$  and $|\omega(p)|=1$ for $p \leq P_1$ such that $\omega(p)=a_p$ for $p \leq N$ and
\begin{gather} \label{iia0}
\max_{1 \leq k \leq n}\abs{ 
\sum_{p < P_1}  \frac{\omega(p)c_k(p)}{ p} -C_k} <\frac{\varepsilon} 9. 
\end{gather}
Let us now choose $B>0$ such that
\begin{gather} \label{ia0}
\max_{1 \leq k \leq n} \max_{s \in K} \abs{\int_0^B g_k(x) e^{-sx} dx-f_k(s)}<\frac \varepsilon 9,
\end{gather}
and  $M \in \Z^+$  sufficiently large 
such that
\begin{gather}
 \label{in1}
 \max_{1 \leq k \leq n} \max_{\substack{|x-y| \leq B/M \\ 0  \leq x,y \leq B}} \max_{s \in K} \abs{e^{-sx} g_k(x)-e^{-sy}g_k(y)} <\frac{\varepsilon} {9B}.
 \end{gather}  
From  \eqref{in1} it follows that the Riemann integral may be estimated by a Riemann sum
\begin{gather} \label{ia1}
\max_{1 \leq k \leq n} \max_{s \in \C} \abs{\sum_{m=1}^M g\p{\frac{mB}M} e^{-smB/M} \frac B M   - \int_0^B g_k(x) e^{-sx} dx}<\frac \varepsilon 9.
  \end{gather}
For a given $\delta$, define 
\begin{gather} \label{Pdef}
 P_2:=\exp \p{\frac B {M \delta}}, \qquad P_3:=\exp\p{\frac {B(M+1)} {M \delta}}. 
\end{gather} 
Now assume that $0<\delta_0 \leq \delta_1$ is sufficiently small so that  if $0 < \delta \leq \delta_0$ then
 \begin{gather} \label{uio}
\max_{s \in K} \max_{1 \leq k \leq n} 
 \sum_{p < P_1}  \abs{\frac{\omega(p)c_k(p)} p   -  \frac{\omega(p)c_k(p)}{ p^{1+\delta s}}}  <\frac{\varepsilon} 9,
\end{gather}
and such that $P_2$ defined by \eqref{Pdef} is sufficiently large so that $P_2 \geq P_1$ and that we may for each $m=1,\ldots,M$ apply Lemma  \ref{lem5} with  $\xi=\frac 1 m$ and $P=P_2^m =\exp(mB/M \delta)$ so that
\begin{gather} \label{ia3}
  \max_{1\leq k \leq n} \abs{ \sum_{mB/M \leq   \delta \log p < (m+1)B/M}\frac{\omega(p) c_k(p)} p -g\p{\frac{mB}M} \frac B M} <\frac{\varepsilon}{9M},
\end{gather}
for some $|\omega(p)|=1$ defined for $mB/M \leq  \delta \log p<(m+1)B/M$. Also assume that $M$ is sufficiently large such that 
\begin{gather} \label{ia3a}
  \max_{s \in K}\max_{1 \leq k \leq n} \sum_{m=1}^M  \, \sum_{mB/M  \leq   \delta \log p < (m+1)B/M} \frac{\abs{c_k(p)}}  p \abs{p^{-\delta s} -e^{-mB/Ms}} <\frac{\varepsilon}{9},
\end{gather}
By   the inequalities  \eqref{ia0}, \eqref{ia1}, \eqref{ia3a} and \eqref{ia3} for each $m=1,\ldots,M$ and the triangle inequality it follows that
\begin{gather} \label{ia4}
         \max_{s \in K} \max_{1\leq k \leq n}\abs{\sum_{P_2 \leq  p < P_3}\frac{\omega(p)c_k(p)} {p^{1+\delta s}} -f_k(s)} <\frac{3\varepsilon}{9}.
\end{gather}
By defining  $\omega(p):=\omega_0(p)$ when $P_1<p\leq P_2$ and when $p>P_3$  it follows that 
\begin{gather} \label{ia5}
       \max_{s \in K} \max_{1\leq k \leq n}\abs{\sum_{P_1 \leq   p < P_2}\frac{\omega(p)c_k(p)} {p^{1+\delta s}}} <\frac{2\varepsilon}{9}, \qquad
  \max_{s \in K} \max_{1\leq k \leq n}\abs{\sum_{P_3 \leq   p}\frac{\omega(p)c_k(p)} {p^{1+\delta s}}} <\frac{\varepsilon}{9},
\end{gather}
         by applying the inequality \eqref{oo3}  twice (combined with the triangle inequality) and once respectively. The conclusion of our lemma follows by the inequalities \eqref{iia0},  \eqref{uio},  \eqref{ia4}, \eqref{ia5} and the triangle inequality.
\qed


\begin{thebibliography}{1}



\bibitem{Andersson9}
Johan Andersson.
\newblock Voronin universality on the abscissa of absolute convergence. 2020, 
\newblock \href{https://arxiv.org/abs/2008.04268}{arXiv:2008.04268 [math.NT]}

\bibitem{Andersson10}
Johan Andersson.
\newblock Universality of the {H}urwitz zeta-function on the half plane of
  absolute convergence. 2020, 
\newblock \href{https://arxiv.org/abs/2008.04709}{arXiv:2008.04709 [math.NT]}

\bibitem{BookerThorne}
Andrew~R. Booker and Frank Thorne.
\newblock Zeros of {$L$}-functions outside the critical strip.
\newblock {\em Algebra Number Theory}, 8(9):2027--2042, 2014.

\bibitem{Lee}
Yoonbok Lee, Takashi Nakamura, and \L{}ukasz Pa\'{n}kowski.
\newblock Selberg's orthonormality conjecture and joint universality of
  {$L$}-functions.
\newblock {\em Math. Z.}, 286(1-2):1--18, 2017.

\bibitem{Righetti0}
Mattia Righetti.
\newblock Zeros of combinations of {E}uler products for {$\sigma>1$}.
\newblock {\em Monatsh. Math.}, 180(2):337--356, 2016.

\bibitem{Righetti}
Mattia Righetti.
\newblock On the density of zeros of linear combinations of {E}uler products
  for {$\sigma >1$}.
\newblock {\em Algebra Number Theory}, 11(9):2131--2163, 2017.

\bibitem{Steuding}
J{\"o}rn Steuding.
\newblock {\em Value-distribution of {$L$}-functions}, volume 1877 of {\em
  Lecture Notes in Mathematics}.
\newblock Springer, Berlin, 2007.

\bibitem{Weyl}
Hermann Weyl.
\newblock {\"Uber ein Problem aus dem Gebiete der diophantischen
  Approximationen.}
\newblock {\em {Nachr. Ges. Wiss. G\"ottingen, Math.-Phys. Kl.}},
  1914:234--244, 1914.

\end{thebibliography}
\end{document}